\documentclass[a4paper,reqno]{amsart}%
\usepackage{amsfonts}
\usepackage{amsmath}
\usepackage{amssymb}
\usepackage{graphicx}%
\providecommand{\U}[1]{\protect\rule{.1in}{.1in}}
\newtheorem{theorem}{Theorem}[section]
\theoremstyle{plain}

\newtheorem{condition}[theorem]{Condition}

\newtheorem{corollary}[theorem]{Corollary}

\newtheorem{example}[theorem]{Example}

\newtheorem{lemma}[theorem]{Lemma}

\newtheorem{proposition}[theorem]{Proposition}

\numberwithin{equation}{section}
\begin{document}
\title[Global boundedness of the fundamental solution]{Global boundedness of the fundamental solution of parabolic equations with
unbounded coefficients}
\author{Esther Bleich}
\address{Karlsruhe Institute of Technology, Department of Mathematics, Institute for
Analysis, 76128 Karlsruhe, Germany }
\email{esther.bleich@partner.kit.edu}
\date{Januar 31, 2013}

\begin{abstract}
The purpose of this paper is to obtain an upper bound for the fundamental
solution for parabolic Cauchy problem $\partial_{t}u=Au$, $u(x,0)=f(x)$, on $%
\mathbb{R}
^{N}\times(0,\infty)$, where $A$ is a second order elliptic partial
differential operator with unbounded coefficients such that its potential and
the potential of the formal adjoint operator $A^{\ast}$ are bounded from below.

\end{abstract}
\maketitle

\section{Introduction}

Let $A$ be a second order elliptic partial differential operator with real
coefficients given by%
\begin{equation}
A=\sum_{i,j=1}^{N}D_{j}\left(  a_{ij}D_{i}\right)  +\sum_{i=1}^{N}F_{i}%
D_{i}-H=A_{0}+F\cdot D-H\text{,} \label{e_Definition operator A}%
\end{equation}
where $A_{0}=\sum_{i,j=1}^{N}D_{j}\left(  a_{ij}D_{i}\right)  $ and $F=\left(
F_{i}\right)  _{i=1,...,N}$. We consider the parabolic Cauchy problem%
\begin{equation}
\left\{
\begin{array}
[c]{ll}%
\partial_{t}u(x,t)=Au(x,t)\text{,} & x\in%
\mathbb{R}
^{N}\text{, }t>0\text{,}\\
u(x,0)=f(x)\text{,} & x\in%
\mathbb{R}
^{N}\text{,}%
\end{array}
\right.  \label{e_Definition parabolic problem}%
\end{equation}
where $f\in C_{b}(%
\mathbb{R}
^{N})$ for $N\in%
\mathbb{N}
$ is given.

It is known that if $a_{ij},D_{j}a_{ij},F_{i},H\in C_{loc}^{\alpha}(%
\mathbb{R}
^{N})$ for all $i,j\in\{1,...,N\}$ and some $\alpha\in(0,1)$ and if
$\inf_{x\in%
\mathbb{R}
^{N}}H(x)>-\infty$, then problem (\ref{e_Definition parabolic problem}) has at
least one solution $u\in C(%
\mathbb{R}
^{N}\times\lbrack0,\infty))\cap C^{2,1}(%
\mathbb{R}
^{N}\times(0,\infty))$ given by%
\[
u(x,t)=\int_{%
\mathbb{R}
^{N}}p(x,y,t)f(y)dy\text{,\qquad}(x,t)\in%
\mathbb{R}
^{N}\times\lbrack0,\infty)\text{,}%
\]
where $p=p(x,y,t)>0$, $(x,y,t)\in%
\mathbb{R}
^{N}\times%
\mathbb{R}
^{N}\times(0,\infty)$, is the fundamental solution (see \cite[Theorem
2.2.5]{bl-buch}).

We assume the following conditions on the coefficients of $A$ which will be
kept without further mentioning.

\begin{condition}
\label{Condition_The main conditions}\hfill

\begin{enumerate}
\item[(i)] $N\geq3$.

\item[(ii)] $a_{ij}\in C_{loc}^{2+\alpha}(%
\mathbb{R}
^{N})$, $F_{i}\in C_{loc}^{1+\alpha}(%
\mathbb{R}
^{N})$, $H\in C_{loc}^{\alpha}(%
\mathbb{R}
^{N})$, $a_{ij}=a_{ji}$ for all $i,j=1,...,N$ and some $\alpha\in(0,1)$.

\item[(iii)] $H(x)\geq H_{0}$ and $\operatorname{div}F(x)+H(x)\geq H_{0}%
^{\ast}$ for each $x\in%
\mathbb{R}
^{N}$, where $H_{0},H_{0}^{\ast}\leq0$.

\item[(iv)] There exists a constant $\lambda>0$ such that%
\begin{equation}
\lambda\left\vert \xi\right\vert ^{2}\leq\sum_{i,j=1}^{N}a_{ij}(x)\xi_{i}%
\xi_{j}\qquad\text{for all }x,\xi\in%
\mathbb{R}
^{N}\text{.} \label{e_ellipticity of Ao}%
\end{equation}

\end{enumerate}
\end{condition}

Notice that the diffusion coefficients $a_{ij},_{i,j=1,...,N}$, the drift
$F=(F_{i})_{i=1,...,N}$ and the potential $H$ are not assumed to be bounded in
$%
\mathbb{R}
^{N}$.

\subsection{The main result}

We prove that under above conditions the fundamental solution $p$ satisfies%
\begin{equation}
p(x,y,t)\leq C_{N,\lambda}e^{\gamma t}t^{-\frac{N}{2}}\text{,}\qquad(x,y,t)\in%
\mathbb{R}
^{N}\times%
\mathbb{R}
^{N}\times(0,\infty)\text{,} \label{e_main result}%
\end{equation}
for the constants%
\begin{equation}
C_{N,\lambda}=\frac{2^{N-1}\Gamma\left(  \frac{N+1}{2}\right)  }{\pi
^{\frac{N+1}{2}}(\lambda(N-2))^{\frac{N}{2}}} \label{e_C_N_lambda}%
\end{equation}
and%
\begin{equation}
\gamma=-\frac{3}{4}\left(  H_{0}^{\ast}+H_{0}\right)  \geq0\text{.}
\label{e_gamma}%
\end{equation}

\subsection{\label{Section_Notation}Notation}

For $x\in%
\mathbb{R}
^{N}$, $\left\vert x\right\vert $ denotes the Euclidean norm. The function
spaces, $L^{q}(\Omega)$ spaces, $1\leq q<\infty$, $\Omega\subseteq%
\mathbb{R}
^{N}$ are always meant with respect to the Lebesgue measure and are endowed
with the usual norm%
\[
\left\Vert \psi\right\Vert _{L^{q}(\Omega)}=\left(  \int_{\Omega}\left\vert
\psi(y)\right\vert ^{q}dy\right)  ^{\frac{1}{q}}%
\]
For $0<\alpha<1$ we denote by $C_{loc}^{k+\alpha}(\Omega)$ the space of all
functions $u$ whose $k^{\text{th}}$ derivatives are locally $\alpha
$-H\"{o}lder continuous. Furthermore, we denote by $C_{loc}^{2+\alpha
,1+\alpha/2}(\Omega\times J)$, where $J\subset\lbrack0,\infty)$ is an
interval, the space of all functions $u$ such that $u,\partial_{t}u,D_{i}u$
and $D_{ij}u$ are locally $\alpha$-H\"{o}lder continuous. $B(x,R)$ denotes the
open ball of $%
\mathbb{R}
^{N}$ of radius $R$ and centre $x$. If $u:%
\mathbb{R}
^{N}\times J\rightarrow%
\mathbb{R}
$, where $J\subset\lbrack0,\infty)$ is an interval, we use the notations%
\[
\partial_{t}u=\frac{\partial u}{\partial t}\text{,\qquad}D_{i}u=\frac{\partial
u}{\partial x_{i}}\text{,\qquad}D_{ij}u=D_{i}D_{j}u\text{,\qquad}Du=\left(
D_{1}u,...,D_{N}u\right)
\]
and%
\[
\left\vert Du\right\vert ^{2}=\sum_{i=1}^{N}\left\vert D_{i}u\right\vert
^{2}\text{.}%
\]
We write $a(\xi,\nu)$ for $\sum_{i,j=1}^{N}a_{ij}(\cdot)\xi_{i}\nu_{j}$ and
$\xi,\nu\in%
\mathbb{R}
^{N}$. It then holds%
\begin{equation}
\left\vert a\left(  \xi,\nu\right)  \right\vert ^{2}\leq a\left(  \xi
,\xi\right)  a\left(  \nu,\nu\right)  \text{\qquad for all }\xi,\nu\in%
\mathbb{R}
^{N}\text{.} \label{e_Schwarz inequality}%
\end{equation}
We further set%
\[
\left\vert a\right\vert ^{2}=\sum_{i,j=1}^{N}a_{ij}^{2}\text{,\qquad
}\left\vert F\right\vert ^{2}=\sum_{i=1}^{N}F_{i}^{2}\text{.}%
\]
Observe that%
\begin{equation}
\left\vert a\left(  \xi,\nu\right)  \right\vert \leq\left\vert a\right\vert
\left\vert \xi\right\vert \left\vert \nu\right\vert \text{\qquad for all }%
\xi,\nu\in%
\mathbb{R}
^{N}\text{.} \label{e_Schwarz inequality 2}%
\end{equation}

We further define a cut-off function $\eta_{n}$. Let $\eta\in C_{c}^{2}(%
\mathbb{R}
^{N})$ be such that $\eta(y)=1$ if $\left\vert y\right\vert \leq1$,
$\eta(y)=0$ if $\left\vert y\right\vert \geq3$, $0\leq\eta\leq1$ and
$\left\vert D\eta\right\vert \leq1${\large .} For each $n\in%
\mathbb{N}
$ we set $\eta_{n}(y):=\eta\left(  \frac{y}{n}\right)  $. Then $\eta
_{n}|_{B(0,n)}=1$, $\eta_{n}|_{%
\mathbb{R}
^{N}\setminus B(0,3n)}=0$ and $0\leq\eta_{n}\leq1$.It follows that%
\begin{equation}
\left\vert D\eta_{n}(y)\right\vert \leq\frac{1}{n}\text{,\qquad for all }y\in%
\mathbb{R}
^{N}\text{ and }n\in%
\mathbb{N}
\text{.} \label{e_properties cut-off-function}%
\end{equation}

If $B$ is a differential operator, then we write $B(Dx)$ (or $B(Dy)$) instead
of $B$ to emphasize that we derive with respect to $x$ (or $y$).

\section{Preliminaries}

\subsection{\label{Section_Construction of p}Construction of $p$}

We briefly recall the construction of a fundamental solution $p$. For more
details we refer to \cite[Chapter 2]{bl-buch} and \cite[Section 4]{Metafune1}
for the case $H=0$. The idea is to consider the Cauchy-Dirichlet problem%
\begin{equation}
\left\{
\begin{array}
[c]{ll}%
\partial_{t}u_{n}(x,t)=Au_{n}(x,t)\text{,} & x\in B\left(  0,n\right)  \text{,
}t>0\text{,}\\
u_{n}\left(  x,t\right)  =0\text{,} & x\in\partial B\left(  0,n\right)
\text{, }t>0\text{,}\\
u_{n}(x,0)=f(x)\text{,} & x\in B\left(  0,n\right)  \text{,}%
\end{array}
\right.  \label{e_Definition parabolic problem on Balls}%
\end{equation}
in the ball $B\left(  0,n\right)  $ for a given $f\in C\left(  \overline
{B(0,n)}\right)  $ and $n\in%
\mathbb{N}
$. By classical results for parabolic Cauchy problems in bounded domains (e.g.
\cite[Chapter III, \S 4]{Friedman}) we know that the problem
(\ref{e_Definition parabolic problem on Balls}) admits a unique solution%
\[
u_{n}\in C(\overline{B(0,n)}\times\lbrack0,\infty))\cap C^{2,1}(B(0,n)\times
(0,\infty))\text{.}%
\]
Moreover, Condition \ref{Condition_The main conditions} implies existence and
uniqueness of a Green function%
\[
0<p_{n}=p_{n}(x,y,t)\in C(B(0,n)\times B(0,n)\times(0,\infty))
\]
such that for each fixed $x\in B(0,n)$ it holds%
\[
p_{n}(x,\cdot,\cdot)\in C_{loc}^{2+\alpha,1+\alpha/2}\left(  B(0,n)\times
(0,\infty)\right)
\]
and for each fixed $y\in B(0,n)$ it holds%
\[
p_{n}\left(  \cdot,y,\cdot\right)  \in C_{loc}^{2+\alpha,1+\alpha/2}\left(
B(0,n)\times(0,\infty)\right)  \text{.}%
\]
Furthermore, for each fixed $y\in B(0,n)$ the function $p_{n}(\cdot,y,\cdot)$
satisfies%
\[
\partial_{t}p_{n}(x,y,t)=A(Dx)p_{n}(x,y,t)
\]
with respect to $(x,t)\in B(0,n)\times(0,\infty)$ and for each fixed $x\in
B(0,n)$ it holds%
\[
\partial_{t}p_{n}(x,y,t)=A^{\ast}(Dy)p_{n}(x,y,t)
\]
with respect to $\left(  y,t\right)  \in B\left(  0,n\right)  \times\left(
0,\infty\right)  $, where%
\begin{equation}
A^{\ast}=A_{0}-F\cdot D-(\operatorname{div}F+H)
\label{e_Definition adjoint operator}%
\end{equation}
is the formal adjoint operator of $A$, such that%
\begin{equation}
p_{n}^{\ast}(y,x,t)=p_{n}(x,y,t) \label{e_p*_n=p_n}%
\end{equation}
is the unique Green function for the problem%
\begin{equation}
\left\{
\begin{array}
[c]{ll}%
\partial_{t}v_{n}(y,t)=A^{\ast}v_{n}(y,t)\text{,} & y\in B(0,n)\text{,
}t>0\text{,}\\
v_{n}(y,t)=0\text{,} & y\in\partial B(0,n)\text{, }t>0\text{,}\\
v_{n}(y,0)=f(y)\text{,} & y\in B(0,n)\text{,}%
\end{array}
\right.  \label{e_Definition adjoint problem on Balls}%
\end{equation}
\ 

The proof of these statements one can find in \cite[Section III,
\S 7]{Friedman}. For the solution $u_{n}$ of Problem
(\ref{e_Definition parabolic problem on Balls}) we hence have%
\[
u_{n}(x,t)=\int_{B(0,n)}p_{n}(x,y,t)f(y)dy
\]
and%
\[
\int_{B(0,n)}p_{n}(x,y,t)f(y)dy\rightarrow f\left(  x\right)  \qquad\text{as
}t\rightarrow0\text{ for each }x\in B(0,n)
\]
and for the solution $v_{n}$ of Problem
(\ref{e_Definition adjoint problem on Balls}) we have%
\[
v_{n}(y,t)=\int_{B\left(  0,n\right)  }p_{n}(x,y,t)f(x)dx
\]
and%
\[
\int_{B\left(  0,n\right)  }p_{n}(x,y,t)f(x)dx\rightarrow f(y)\qquad\text{as
}t\rightarrow0\text{ for each }y\in B(0,n)\text{.}%
\]

Using the classical maximum principle, one obtains that the sequence $\left(
p_{n}\right)  $ is increasing with respect to $n\in%
\mathbb{N}
$. So we extend each function $p_{n}$ to $%
\mathbb{R}
^{N}\times%
\mathbb{R}
^{N}\times(0,+\infty)$ with value zero for $x,y\in%
\mathbb{R}
^{N}\setminus B(0,n)$ and still denote by $p_{n}$ the so obtained function. It
then holds%
\begin{equation}
p_{n}(x,y,t)\leq p_{n+1}(x,y,t) \label{e_p_n < p_n+1 < p}%
\end{equation}
for all $(x,y,t)\in%
\mathbb{R}
^{N}\times%
\mathbb{R}
^{N}\times(0,\infty)$ and $n\in%
\mathbb{N}
$. One sets%
\begin{equation}
p\left(  x,y,t\right)  =\lim_{n\rightarrow\infty}p_{n}\left(  x,y,t\right)
\text{,\qquad pointwise for }\left(  x,y,t\right)  \in%
\mathbb{R}
^{N}\times%
\mathbb{R}
^{N}\times\left(  0,\infty\right)  \text{.} \label{e_p=lim p_n}%
\end{equation}

\subsection{Properties of $p$}

We formulate the main properties of $p$ in the following proposition. The
proof one can find in \cite[Chapter 2]{bl-buch} and in \cite{Metafune1} for
the case $H=0$ (see also \cite{meine Diss}).

\begin{proposition}
\label{Proposition_properties of p}Under assumptions of Condition
\ref{Condition_The main conditions} the following statements hold.

\begin{enumerate}
\item[(i)] $\int_{%
\mathbb{R}
^{N}}p(x,y,t)dy\leq e^{-H_{0}t}$ for all $(x,t)\in%
\mathbb{R}
^{N}\times(0,\infty)$.

\item[(ii)] $0<p(x,y,t+s)=\int_{%
\mathbb{R}
^{N}}p(x,z,t)p(z,y,s)dz$ for all $x,y\in%
\mathbb{R}
^{N}$ and $s,t>0$.

\item[(iii)] For each fixed $y\in%
\mathbb{R}
^{N}$ it holds $\partial_{t}p(x,y,t)=A(Dx)p(x,y,t)$ for all $(x,t)\in%
\mathbb{R}
^{N}\times(0,\infty)$.

\item[(iv)] For each fixed $x\in%
\mathbb{R}
^{N}$ it holds $\partial_{t}p(x,y,t)=A^{\ast}(Dy)p(x,y,t)$ for all $(y,t)\in%
\mathbb{R}
^{N}\times(0,\infty)$.

\item[(v)] $u(x,t)=\int_{%
\mathbb{R}
^{N}}p(x,y,t)f(y)dy$ solves for each $f\in C_{b}(%
\mathbb{R}
^{N})$ problem (\ref{e_Definition parabolic problem}), $u\in C(%
\mathbb{R}
^{N}\times\lbrack0,\infty))\cap C_{loc}^{2+\alpha,1+\alpha/2}(%
\mathbb{R}
^{N}\times(0,\infty))$ and it holds%
\[
\left\vert u(x,t)\right\vert \leq e^{-H_{0}t}\left\Vert f\right\Vert _{\infty
}\qquad\text{for all }(x,t)\in%
\mathbb{R}
^{N}\times\lbrack0,\infty)\text{.}%
\]

\item[(vi)] $v(y,t)=\int_{%
\mathbb{R}
^{N}}p(x,y,t)f(x)dx$ solves for each $f\in C_{b}(%
\mathbb{R}
^{N})$ problem%
\begin{equation}
\left\{
\begin{array}
[c]{ll}%
\partial_{t}v(y,t)=A^{\ast}v(y,t)\text{,} & y\in%
\mathbb{R}
^{N}\text{, }t>0\text{,}\\
v(y,0)=f(y)\text{,} & y\in%
\mathbb{R}
^{N}\text{,}%
\end{array}
\right.  \label{e_Definition adjoint problem}%
\end{equation}
$v\in C(%
\mathbb{R}
^{N}\times\lbrack0,\infty))\cap C_{loc}^{2+\alpha,1+\alpha/2}(%
\mathbb{R}
^{N}\times(0,\infty))$ and it holds%
\[
\left\vert v(y,t)\right\vert \leq e^{-H_{0}^{\ast}t}\left\Vert f\right\Vert
_{\infty}\qquad\text{for all }(y,t)\in%
\mathbb{R}
^{N}\times\lbrack0,\infty)\text{.}%
\]

\item[(vii)] For any bounded Borel function $f\geq0$ with $f\not \equiv 0$ it
holds%
\[
\int_{%
\mathbb{R}
^{N}}p(x,y,t)f(y)dy>0\qquad\text{for all }(x,t)\in%
\mathbb{R}
^{N}\times(0,\infty)
\]
and%
\[
\int_{%
\mathbb{R}
^{N}}p(x,y,t)f(x)dx>0\qquad\text{for all }(y,t)\in%
\mathbb{R}
^{N}\times(0,\infty)
\]

(\emph{positivity}).
\end{enumerate}
\end{proposition}

The global boundedness of $p$ was studied for example in \cite{Trans},
\cite{Trans new} for the case of bounded diffusion coefficients $a_{ij}$,
$i,j=1,...,N$, and in \cite{meine Diss} for the general case. It was assumed
the existence of some Lyapunov function $1\leq V\in C^{2}(%
\mathbb{R}
^{N})$, that is%
\[
\lim_{\left\vert x\right\vert \rightarrow\infty}V(x)=\infty\qquad
\text{and}\qquad AV(x)\leq kV(x)\qquad\text{for all }x\in%
\mathbb{R}
^{N}%
\]
and some constant $k>-H_{0}$. Moreover, the coefficients of $A$ must growth
not faster as $V^{\frac{1}{N+1}}$. We remark that the existence of a Lyapunov
function yields the uniqueness of the bounded solution of Problem
(\ref{e_Definition parabolic problem}).

\medskip

The current case allows the nonuniqueness of the bounded solution and
arbitrary grow of the coefficients of $A$.

\medskip

The similar result one can find in \cite{Metafune new} under assumption of
bounded diffusion coefficients. Therefore the technics from \cite{Metafune
new} are unsuitable in the current case.

\section{Global boundedness of the fundamental solution}

From classical theory we know that if the operator $A$ has bounded
coefficients, then it holds%
\[
p(x,y,t)\leq Ct^{-\frac{N}{2}}\qquad\text{for all }(x,y,t)\in%
\mathbb{R}
^{N}\times%
\mathbb{R}
^{N}\times(0,\infty)
\]
for some constant $C>0$, depending on the supremum norm of coefficients of the
operator $A$ (see e. g. \cite[Chapter I, (6.12)]{Friedman}).

We will approximate the operator $A$ by operators%
\[
A^{(m)}=A_{0}^{(m)}+F^{(m)}\cdot D-H^{(m)}\text{,}\qquad m\in%
\mathbb{N}
\text{.}%
\]
Therefor, for $m\in%
\mathbb{N}
$ we set%
\[
a_{ij}^{\left(  m\right)  }=\eta_{m}a_{ij}+\lambda\left(  1-\eta_{m}\right)
\delta_{ij}\text{,}%
\]
where $\delta_{ii}=1$ and $\delta_{ij}=0$ if $i\neq j$, a constant $\lambda>0$
is given as in (\ref{e_ellipticity of Ao}) and the cut-off function $\eta_{m}$
is given as in Section \ref{Section_Notation}. Furthermore, we set%
\[
A_{0}^{(m)}=\sum_{i,j=1}^{N}D_{i}(a_{ij}^{(m)}D_{j})\text{,\qquad}F_{i}%
^{(m)}=\eta_{m}F_{i}%
\]
and%
\[
H^{(m)}=\eta_{m}H-F\cdot D\eta_{m}+\left\vert F\right\vert \left\vert
D\eta_{m}\right\vert \text{.}%
\]

We then obtain that the coefficients of $A^{(m)}$ are bounded and it holds%
\begin{equation}
a^{(m)}(\cdot)(\xi,\xi):=\sum_{i,j=1}^{N}a_{ij}^{(m)}(\cdot)\xi_{i}\xi_{j}%
\geq\lambda\left\vert \xi\right\vert ^{2}\text{.} \label{e_m-ellipticity}%
\end{equation}
Thus $A^{(m)}$ is elliptic. Moreover, we have%
\begin{equation}
H^{(m)}(x)\geq\eta_{m}(x)H(x)\geq H_{0} \label{e_H_o^(m)}%
\end{equation}
and%
\begin{equation}
\operatorname{div}F^{(m)}(x)+H^{(m)}(x)\geq\eta_{m}(x)(\operatorname{div}%
F(x)+H(x))\geq H_{0}^{\ast}\text{.} \label{e_H_o^ * ^(m)}%
\end{equation}
Let $p^{(m)}=p^{(m)}(x,y,t)$ be the fundamental solution for $A^{(m)}$. It
then holds%
\begin{equation}
\partial_{t}p^{(m)}=A_{0}^{(m)}(Dy)p^{(m)}-F^{(m)}\cdot Dp^{(m)}%
-(\operatorname{div}F^{(m)}+H^{(m)})p^{(m)} \label{e_equality for p_m}%
\end{equation}
with respect to $(y,t)\in%
\mathbb{R}
^{N}\times(0,\infty)$ for each fixed $x\in%
\mathbb{R}
^{N}$. In the next lemma we present some estimate of $L^{2}(%
\mathbb{R}
^{N})$ norm of $p^{(m)}$, $m\in%
\mathbb{N}
$. The calculation method was presented by John Nash in \cite{Nash} for the
case $F=0$, $H=0$ and $a_{ij}\in C_{b}^{1}(%
\mathbb{R}
^{N})$, $i,j=1,...,N$. In the proof a special case of
Gagliardo--Nirenberg--Sobolev inequality (see \cite{Talenti}) will be used%
\begin{equation}
S\left(  \int_{%
\mathbb{R}
^{N}}\left\vert u(x)\right\vert ^{\frac{2N}{N-2}}dx\right)  ^{\frac{N-2}{N}%
}\leq\int_{%
\mathbb{R}
^{N}}\left\vert Du(x)\right\vert ^{2}dx\text{,}
\label{e_Gagliardo-Nierenberg-Sobolev inequality}%
\end{equation}
where the constant $S$ is given by%
\begin{equation}
S=\frac{4^{\frac{N-1}{N}}\pi^{\frac{N+1}{N}}N(N-2)}{\Gamma\left(  \frac
{N+1}{2}\right)  ^{\frac{2}{N}}}\text{.} \label{e_S}%
\end{equation}

\begin{lemma}
\label{Lemma}For each $m\in%
\mathbb{N}
$ it holds%
\begin{equation}
\int_{%
\mathbb{R}
^{N}}p^{(m)}\left(  x,y,t\right)  ^{2}dy\leq Ce^{\gamma_{1}t}t^{-\frac{N}{2}}
\label{e_L^2_m}%
\end{equation}
for all $\left(  x,t\right)  \in%
\mathbb{R}
^{N}\times\left(  0,\infty\right)  $, where%
\[
C=\frac{2^{\frac{N-2}{2}}\Gamma(\frac{N+1}{2})}{\pi^{\frac{N+1}{2}}\left(
\lambda(N-2)\right)  ^{\frac{N}{2}}}%
\]
and%
\[
\gamma_{1}=-H_{0}^{\ast}-2H_{0}\geq0\text{.}%
\]

\end{lemma}

\begin{proof}
We fix arbitrary $x\in%
\mathbb{R}
^{N}$ and $m\in%
\mathbb{N}
$. For each $n\in%
\mathbb{N}
$ we set%
\begin{equation}
\zeta_{n}(x,t)=\int_{%
\mathbb{R}
^{N}}\eta_{n}(y)^{2}p^{(m)}(x,y,t)^{2}dy\text{,\qquad}t\in(0,\infty)\text{.}
\label{e_Zetha_n}%
\end{equation}

Since (\ref{e_equality for p_m}) for any $t\in\left(  0,\infty\right)  $, it
holds%
\begin{align*}
\partial_{t}\zeta_{n}  &  =\int_{%
\mathbb{R}
^{N}}2\eta_{n}^{2}p^{(m)}\partial_{t}p^{(m)}dy\\
&  =\int_{%
\mathbb{R}
^{N}}2\eta_{n}^{2}p^{(m)}A_{0}^{\left(  m\right)  }(Dy)p^{(m)}dy-\int_{%
\mathbb{R}
^{N}}2\eta_{n}^{2}p^{(m)}F^{(m)}\cdot Dp^{(m)}dy\\
&  -\int_{%
\mathbb{R}
^{N}}2\eta_{n}^{2}(p^{(m)})^{2}\left(  \operatorname{div}F^{(m)}%
+H^{(m)}\right)  dy\text{.}%
\end{align*}
Integration by parts yields%
\begin{align}
-\partial_{t}\zeta_{n}  &  =\int_{%
\mathbb{R}
^{N}}2\eta_{n}^{2}a^{(m)}(Dp^{(m)},Dp^{(m)})dy\nonumber\\
&  +\int_{%
\mathbb{R}
^{N}}4\eta_{n}p^{(m)}a^{(m)}(D\eta_{n},Dp^{(m)})dy\nonumber\\
&  -\int_{%
\mathbb{R}
^{N}}2\eta_{n}(p^{(m)})^{2}\eta_{m}F\cdot D\eta_{n}dy\nonumber\\
&  +\int_{%
\mathbb{R}
^{N}}\eta_{n}^{2}(p^{(m)})^{2}\eta_{m}(\operatorname{div}F+H)dy+\int_{%
\mathbb{R}
^{N}}\eta_{n}^{2}(p^{(m)})^{2}\eta_{m}Hdy\nonumber\\
&  +\int_{%
\mathbb{R}
^{N}}\eta_{n}^{2}(p^{(m)})^{2}\left(  2\left\vert F\right\vert \left\vert
D\eta_{m}\right\vert -F\cdot D\eta_{m}\right)  dy \label{ee_1}%
\end{align}
Moreover, it holds%
\begin{align*}
\int_{%
\mathbb{R}
^{N}}4\eta_{n}p^{(m)}a^{(m)}(D\eta_{n},Dp^{(m)})dy  &  =\int_{%
\mathbb{R}
^{N}}2a^{(m)}(D(\eta_{n}p^{(m)}),D(\eta_{n}p^{(m)}))dy\\
&  -\int_{%
\mathbb{R}
^{N}}2(p^{(m)})^{2}a^{(m)}(D\eta_{n},D\eta_{n})dy\\
&  -\int_{%
\mathbb{R}
^{N}}2\eta_{n}^{2}a^{(m)}(Dp^{(m)},Dp^{(m)})dy\text{.}%
\end{align*}
Applying this identity to (\ref{ee_1}), we obtain%
\begin{align}
-\partial_{t}\zeta_{n}  &  =\int_{%
\mathbb{R}
^{N}}2a^{(m)}(D(\eta_{n}p^{(m)}),D(\eta_{n}p^{(m)}))dy\nonumber\\
&  -\int_{%
\mathbb{R}
^{N}}2(p^{(m)})^{2}a^{(m)}(D\eta_{n},D\eta_{n})dy\nonumber\\
&  -\int_{%
\mathbb{R}
^{N}}2\eta_{n}(p^{(m)})^{2}\eta_{m}F\cdot D\eta_{n}dy\nonumber\\
&  +\int_{%
\mathbb{R}
^{N}}\eta_{n}^{2}(p^{(m)})^{2}\eta_{m}\left(  \operatorname{div}F+H\right)
dy+\int_{%
\mathbb{R}
^{N}}\eta_{n}^{2}(p^{(m)})^{2}\eta_{m}Hdy\nonumber\\
&  +\int_{%
\mathbb{R}
^{N}}\eta_{n}^{2}(p^{(m)})^{2}\left(  2\left\vert F\right\vert \left\vert
D\eta_{m}\right\vert -F\cdot D\eta_{m}\right)  dy\text{.} \label{ee_2}%
\end{align}

We fix an arbitrary $t\in(0,\infty)$. We estimate the terms of (\ref{ee_2}).
Using (\ref{e_m-ellipticity}), Proposition \ref{Proposition_properties of p}
and (\ref{e_properties cut-off-function}), we obtain%
\[
\int_{%
\mathbb{R}
^{N}}2a^{(m)}(D(\eta_{n}p^{(m)}),D(\eta_{n}p^{(m)}))dy\geq\int_{%
\mathbb{R}
^{N}}2\lambda\left\vert D(\eta_{n}p^{(m)})\right\vert ^{2}dy\text{,}%
\]

\begin{align*}
-\int_{%
\mathbb{R}
^{N}}2(p^{(m)})^{2}a^{(m)}(D\eta_{n},D\eta_{n})dy  &  \geq-\int_{%
\mathbb{R}
^{N}}\frac{2}{n^{2}}(p^{(m)})^{2}\left\vert a^{(m)}\right\vert dy\\
&  \geq-\frac{2}{n}\left\Vert p^{(m)}(x,\cdot,t)\right\Vert _{\infty
}\left\Vert \left\vert a^{(m)}\right\vert \right\Vert _{\infty}e^{-H_{0}%
t}\text{,}%
\end{align*}%
\begin{align*}
-\int_{%
\mathbb{R}
^{N}}2\eta_{n}(p^{(m)})^{2}\eta_{m}F\cdot D\eta_{n}dy  &  \geq-\int_{%
\mathbb{R}
^{N}}\frac{2}{n}(p^{(m)})^{2}\eta_{m}\left\vert F\right\vert dy\\
&  \geq-\frac{2}{n}\left\Vert p^{(m)}(x,\cdot,t)\right\Vert _{\infty
}\left\Vert \left\vert F^{(m)}\right\vert \right\Vert _{\infty}e^{-H_{0}%
t}\text{,}%
\end{align*}

\begin{align*}
\int_{%
\mathbb{R}
^{N}}\eta_{n}^{2}(p^{(m)})^{2}  &  \eta_{m}\left(  \operatorname{div}%
F+H\right)  dy+\int_{%
\mathbb{R}
^{N}}\eta_{n}^{2}(p^{(m)})^{2}\eta_{m}Hdy\\
&  \geq\left(  H_{0}^{\ast}+H_{0}\right)  \zeta_{n}%
\end{align*}
and%
\[
\int_{%
\mathbb{R}
^{N}}\eta_{n}^{2}(p^{(m)})^{2}\left(  2\left\vert F\right\vert \left\vert
D\eta_{m}\right\vert -F\cdot D\eta_{m}\right)  dy\geq0\text{.}%
\]

\noindent We set%
\[
\theta=H_{0}^{\ast}+H_{0}\leq0
\]
\noindent Hence, from (\ref{ee_2}) it follows%
\begin{equation}
-\partial_{t}\zeta_{n}\geq\int_{%
\mathbb{R}
^{N}}2\lambda\left\vert D(\eta_{n}p^{(m)})\right\vert ^{2}dy+\theta\zeta
_{n}-\omega_{n}\text{,} \label{ee_5}%
\end{equation}
where%
\[
\omega_{n}=\omega_{n}(x,t)=\frac{2}{n}e^{-H_{0}t}\left\Vert p^{(m)}%
(x,\cdot,t)\right\Vert _{\infty}\left(  \left\Vert \left\vert a^{(m)}%
\right\vert \right\Vert _{\infty}+\left\Vert \left\vert F^{(m)}\right\vert
\right\Vert _{\infty}\right)  \text{.}%
\]
Moreover, $0\leq\omega_{n}(x,t)\rightarrow0$ as $n\rightarrow\infty$ for any
$(x,t)\in%
\mathbb{R}
^{N}\times(0,\infty)$. Furthermore, the Gagliardo--Nirenberg--Sobolev
inequality (\ref{e_Gagliardo-Nierenberg-Sobolev inequality}) implies%
\begin{equation}
\int_{%
\mathbb{R}
^{N}}\left\vert D\left(  \eta_{n}p^{(m)}\right)  \right\vert ^{2}dy\geq
S\left(  \int_{%
\mathbb{R}
^{N}}\left(  \eta_{n}p^{(m)}\right)  ^{\frac{2N}{N-2}}dy\right)  ^{\frac
{N-2}{N}} \label{ee_4}%
\end{equation}
for the Sobolev constant $S=S\left(  N\right)  $ given in (\ref{e_S}). Since%
\[
0<\int_{%
\mathbb{R}
^{N}}\eta_{1}p^{(m)}dy\leq\int_{%
\mathbb{R}
^{N}}\eta_{n}p^{(m)}dy\leq\int_{%
\mathbb{R}
^{N}}p^{(m)}dy\leq e^{-H_{0}t}\text{,}%
\]
it holds%
\[
0<e^{H_{0}t}\leq\frac{1}{\int_{%
\mathbb{R}
^{N}}\eta_{n}p^{(m)}dy}<\infty\text{.}%
\]
For $r>1$, this fact leads to%
\begin{align*}%
\biggl
(\int_{%
\mathbb{R}
^{N}}\left(  \eta_{n}p^{(m)}\right)   &  ^{\frac{2N}{N-2}}dy%
\biggr
)^{\frac{1}{r}}\\
=  &  \left(  \int_{%
\mathbb{R}
^{N}}\left(  \left(  \eta_{n}(p^{(m)})\right)  ^{\frac{2N}{\left(  N-2\right)
r}}\right)  ^{r}dy\right)  ^{\frac{1}{r}}\left(  \int_{%
\mathbb{R}
^{N}}\left(  \left(  \eta_{n}p^{(m)}\right)  ^{\frac{r-1}{r}}\right)
^{\frac{r}{r-1}}dy\right)  ^{\frac{r-1}{r}}\\
&  \cdot\left(  \frac{1}{\int_{%
\mathbb{R}
^{N}}\eta_{n}p^{(m)}dy}\right)  ^{\frac{r-1}{r}}\\
\geq &  \left\Vert \left(  \eta_{n}p^{(m)}\right)  ^{\frac{2N}{\left(
N-2\right)  r}}\right\Vert _{r}\left\Vert \left(  \eta_{n}p^{(m)}\right)
^{\frac{r-1}{r}}\right\Vert _{\frac{r}{r-1}}e^{H_{0}\frac{r-1}{r}t}\text{.}%
\end{align*}
H\"{o}lder's inequality then yields%
\begin{equation}
\left(  \int_{%
\mathbb{R}
^{N}}\left(  \eta_{n}p^{(m)}\right)  ^{\frac{2N}{N-2}}dy\right)  ^{\frac{1}%
{r}}\geq\left\Vert \left(  \eta_{n}p^{(m)}\right)  ^{\frac{2N}{\left(
N-2\right)  r}+\frac{r-1}{r}}\right\Vert _{1}e^{H_{0}\frac{r-1}{r}t}\text{.}
\label{ee_3}%
\end{equation}
Choosing $r=\frac{N+2}{N-2}$ in (\ref{ee_3}), we infer%
\[
\left(  \int_{%
\mathbb{R}
^{N}}\left(  \eta_{n}p^{(m)}\right)  ^{\frac{2N}{N-2}}dy\right)  ^{\frac
{N-2}{N+2}}\geq\left\Vert \eta_{n}^{2}(p^{(m)})^{2}\right\Vert _{1}%
e^{\frac{4H_{0}}{N+2}t}=\zeta_{n}e^{\frac{4H_{0}}{N+2}t}%
\]
and hence%
\[
\left(  \int_{%
\mathbb{R}
^{N}}\left(  \eta_{n}p^{(m)}\right)  ^{\frac{2N}{N-2}}dy\right)  ^{\frac
{N-2}{N}}\geq\zeta_{n}^{1+\frac{2}{N}}e^{\frac{4H_{0}}{N}t}\text{.}%
\]

We combine the above inequality with (\ref{ee_4}) and arrive at%
\begin{equation}
\int_{%
\mathbb{R}
^{N}}\left\vert D\left(  \eta_{n}p^{(m)}\right)  \right\vert ^{2}dy\geq
S\zeta_{n}^{1+\frac{2}{N}}e^{\frac{4H_{0}}{N}t}\text{.} \label{e_0030}%
\end{equation}
It then follows from (\ref{ee_5})%
\[
-\partial_{t}\zeta_{n}\geq2\lambda S\zeta_{n}^{1+\frac{2}{N}}e^{\frac{4H_{0}%
}{N}t}+\theta\zeta_{n}-\omega_{n}%
\]
and hence%
\[
-\partial_{t}\left(  e^{\theta t}\zeta_{n}\right)  \geq2\lambda S\zeta
_{n}^{1+\frac{2}{N}}e^{\theta t}e^{\frac{4H_{0}}{N}t}-e^{\theta t}\omega
_{n}\text{.}%
\]

We remark that for $n\in%
\mathbb{N}
$ it holds%
\begin{equation}
0<\delta=\delta(x,t):=\int_{%
\mathbb{R}
^{N}}p^{(m)}(x,y,t)\cdot\eta_{1}(y)^{2}p^{(m)}(x,y,t)dy\leq\zeta
_{n}(x,t)<\infty\label{ee_delta}%
\end{equation}

Taking into account (\ref{ee_delta}), we conclude%
\begin{equation}
\partial_{t}\left(  (e^{\theta t}\zeta_{n})^{-\frac{2}{N}}\right)  \geq
\frac{4\lambda S}{N}e^{-\frac{2\theta}{N}t}e^{\frac{4H_{0}}{N}t}-\frac{2}%
{N}\delta^{-1-\frac{2}{N}}e^{-\frac{2\theta}{N}t}\omega_{n}\text{.}
\label{ee_6}%
\end{equation}

Let further $t_{0}>0$ be such that $2t_{0}<t$. We define $\tau\in C^{\infty}(%
\mathbb{R}
)$ by $0\leq\tau\leq1$, $\tau(s)=0$ for $0\leq s\leq t_{0}$, $\tau(s)=1$ for
$s\geq2t_{0}$ and $\tau^{\prime}\geq0$. We multiply (\ref{ee_6}) by $\tau$ and
get%
\begin{align}
\partial_{t}\left(  \tau(t)(e^{\theta t}\zeta_{n}(x,t))^{-\frac{2}{N}}\right)
&  \geq\frac{4\lambda S}{N}\tau(t)e^{-\frac{2\theta}{N}t}e^{\frac{4H_{0}}{N}%
t}-\frac{2}{N}\tau(t)\delta^{-1-\frac{2}{N}}(x,t)e^{-\frac{2\theta}{N}t}%
\omega_{n}(x,t)\nonumber\\
&  +\tau^{\prime}(t)(e^{\theta t}\zeta_{n}(x,t))^{-\frac{2}{N}}\text{,}
\label{ee_9}%
\end{align}
where the last term on the right side is nonnegative. We set%
\[
\nu_{n}(x,t)=\delta(x,t)^{-1-\frac{2}{N}}e^{-\frac{4H_{0}}{N}t}\omega
_{n}(x,t)\text{.}%
\]
From (\ref{ee_9}) we conclude%
\[
\partial_{t}\left(  \tau(e^{\theta t}\zeta_{n})^{-\frac{2}{N}}\right)
\geq\frac{2}{N}\tau e^{-\frac{2\theta}{N}t}e^{\min\{0,H_{0}\}\frac{4}{N}%
t}\left(  2\lambda S-\nu_{n}\right)  \text{.}%
\]
Since $v_{n}(x,t)\rightarrow0$ as $n\rightarrow\infty$ for any $(x,t)\in%
\mathbb{R}
^{N}\times(0,\infty)$, we can chose $n_{0}\in%
\mathbb{N}
$ such that $2\lambda S-\nu_{n}\geq\lambda S$ for each $n\geq n_{0}$. For such
$n$ we obtain%
\[
\partial_{t}\left(  \tau(e^{\theta t}\zeta_{n})^{-\frac{2}{N}}\right)
\geq\frac{2\lambda S}{N}\tau e^{-\frac{2\theta}{N}t}e^{\frac{4H_{0}}{N}%
t}\text{.}%
\]

Integration from $t_{0}$ to $t$ yields%
\begin{align}
(e^{\theta t}\zeta_{n}(x,t))^{-\frac{2}{N}}  &  \geq\frac{2\lambda S}{N}%
\int_{t_{0}}^{t}\tau(s)e^{-\frac{2\theta}{N}s}e^{\frac{4H_{0}}{N}s}ds\geq
\frac{2\lambda S}{N}\int_{2t_{0}}^{t}e^{-\frac{2\theta}{N}s}e^{\frac{4H_{0}%
}{N}s}ds\nonumber\\
&  \geq\frac{2\lambda S}{N}e^{\frac{2H_{0}}{N}t}(t-2t_{0})\text{.}
\label{ee_7}%
\end{align}
For $n\geq n_{0}$ from (\ref{ee_7}) we deduce%
\begin{equation}
\zeta_{n}(x,t)^{-1}\geq\left(  \frac{2\lambda S}{N}\right)  ^{\frac{N}{2}%
}e^{(H_{0}^{\ast}+2H_{0})t}\left(  t-2t_{0}\right)  ^{\frac{N}{2}}\text{.}
\label{ee_8}%
\end{equation}

Since $\zeta_{n}(x,t)>0$ for any $(x,t)\in%
\mathbb{R}
^{N}\times(0,\infty)$, we obtain from (\ref{ee_8})%
\[
\zeta_{n}\leq\left(  \frac{N}{2\lambda S}\right)  ^{\frac{N}{2}}e^{\gamma
_{1}t}(t-2t_{0})^{-\frac{N}{2}}\text{,}%
\]
where $\gamma_{1}=-H_{0}^{\ast}-2H_{0}\geq0$.

Letting $n\rightarrow\infty,$ Fatou's lemma implies%
\[
\int_{%
\mathbb{R}
^{N}}p^{(m)}\left(  x,y,t\right)  ^{2}dy\leq\left(  \frac{N}{2\lambda
S}\right)  ^{\frac{N}{2}}e^{\gamma_{1}t}\left(  t-2t_{0}\right)  ^{-\frac
{N}{2}}%
\]

Since $t_{0}>0$ can be arbitrary close to $0$ and $(x,t)\in%
\mathbb{R}
^{N}\times(0,\infty)$ are arbitrary, we deduce%
\[
\int_{%
\mathbb{R}
^{N}}p^{(m)}\left(  x,y,t\right)  ^{2}dy\leq\left(  \frac{N}{2\lambda
S}\right)  ^{\frac{N}{2}}e^{\gamma_{1}t}t^{-\frac{N}{2}}%
\]
for all $(x,t)\in%
\mathbb{R}
^{N}\times(0,\infty)$. Using (\ref{e_S}), we then observe%
\[
\int_{%
\mathbb{R}
^{N}}p^{(m)}\left(  x,y,t\right)  ^{2}dy\leq Ce^{\gamma_{1}t}t^{-\frac{N}{2}}%
\]
for all $(x,t)\in%
\mathbb{R}
^{N}\times(0,\infty)$ and each $m\in%
\mathbb{N}
$.
\end{proof}

\medskip

The next step is to show that estimate (\ref{e_L^2_m}) is true for $p$ instead
of $p^{(m)}$. Therefore, we recall the construction of $p^{(m)}$. For fixed
$m\in%
\mathbb{N}
$ we consider the parabolic Cauchy problem%
\begin{equation}
\left\{
\begin{array}
[c]{ll}%
\partial_{t}u_{n}(x,t)=A^{(m)}u_{n}(x,t)\text{,} & x\in B\left(  0,n\right)
\text{, }t>0\text{,}\\
u_{n}\left(  x,t\right)  =0\text{,} & x\in\partial B\left(  0,n\right)
\text{, }t>0\text{,}\\
u_{n}(x,0)=f(x)\text{,} & x\in B\left(  0,n\right)  \text{,}%
\end{array}
\right.  \label{e_Definition parabolic problem on Balls_m}%
\end{equation}
for $f\in C(\overline{B(0,n)})$ and $n\in%
\mathbb{N}
$. We denote by $p_{n}^{(m)}$ the Green function for the problem
(\ref{e_Definition parabolic problem on Balls_m}). We remark that from
(\ref{e_p_n < p_n+1 < p}) and (\ref{e_p=lim p_n}) it follows that%
\[
p_{n}^{(m)}(x,y,t)\leq p^{(m)}(x,y,t)\text{,}\qquad(x,y,t)\in%
\mathbb{R}
^{N}\times%
\mathbb{R}
^{N}\times(0,\infty)\text{,}%
\]
for each $n\in%
\mathbb{N}
$. Note that we consider extended $p_{n}^{(m)}$ on $%
\mathbb{R}
^{N}\times%
\mathbb{R}
^{N}\times(0,\infty)$ with $p_{n}^{(m)}(x,y,t)=0$ for $x,y\in%
\mathbb{R}
^{N}\setminus B\left(  0,n\right)  $ as in Section
\ref{Section_Construction of p}. Since $A^{(m)}=A$ on $B(0,m)$, we deduce that
$p_{m}^{(m)}=p_{m}$, where $p_{m}$ is the Green function for the problem%
\[
\left\{
\begin{array}
[c]{ll}%
\partial_{t}u(x,t)=Au(x,t)\text{,} & x\in B(0,m)\text{, }t>0\text{,}\\
u(x,t)=0\text{,} & x\in\partial B(0,m)\text{, }t>0\text{,}\\
u(x,0)=f(x)\text{,} & x\in B(0,m)\text{.}%
\end{array}
\right.
\]
So we obtain%
\[
p_{m}(x,y,t)\leq p^{(m)}(x,y,t)\text{,}\qquad(x,y,t)\in%
\mathbb{R}
^{N}\times%
\mathbb{R}
^{N}\times(0,\infty)\text{,}%
\]
for each $m\in%
\mathbb{N}
$. Thus Lemma \ref{Lemma} yields%
\[
\int_{%
\mathbb{R}
^{N}}p_{m}\left(  x,y,t\right)  ^{2}dy\leq Ce^{\gamma_{1}t}t^{-\frac{N}{2}}%
\]
for all $(x,y,t)\in%
\mathbb{R}
^{N}\times%
\mathbb{R}
^{N}\times(0,\infty)$ and $m\in%
\mathbb{N}
$, where constants $C$ and $\gamma_{1}$ are given as in Lemma \ref{Lemma}.
Using (\ref{e_p=lim p_n}) and Fatou's lemma we conclude that%
\[
\int_{%
\mathbb{R}
^{N}}p\left(  x,y,t\right)  ^{2}dy\leq Ce^{\gamma_{1}t}t^{-\frac{N}{2}}%
\]
for all $(x,t)\in%
\mathbb{R}
^{N}\times(0,\infty)$. Applying this estimate to the adjoint problem
(\ref{e_Definition adjoint problem}), we obtain%
\[
\int_{%
\mathbb{R}
^{N}}p\left(  x,y,t\right)  ^{2}dx\leq Ce^{\gamma_{2}t}t^{-\frac{N}{2}}%
\]
for all $(y,t)\in%
\mathbb{R}
^{N}\times(0,\infty)$, where%
\begin{equation}
\gamma_{2}=-2H_{0}^{\ast}-H_{0}\geq0\text{.} \label{e_Gamma_2}%
\end{equation}
We formulate this result in the following corollary.

\begin{corollary}
\label{Corollary}Under assumptions of condition
\ref{Condition_The main conditions} it holds%
\[
\int_{%
\mathbb{R}
^{N}}p\left(  x,y,t\right)  ^{2}dy\leq Ce^{\gamma_{1}t}t^{-\frac{N}{2}}%
\qquad\text{for all }(x,t)\in%
\mathbb{R}
^{N}\times(0,\infty)
\]
and%
\[
\int_{%
\mathbb{R}
^{N}}p\left(  x,y,t\right)  ^{2}dx\leq Ce^{\gamma_{2}t}t^{-\frac{N}{2}}%
\qquad\text{for all }(y,t)\in%
\mathbb{R}
^{N}\times(0,\infty)\text{,}%
\]
where $\gamma_{1}$ and $C$ are given as in Lemma \ref{Lemma} and $\gamma_{2}$
is given as in (\ref{e_Gamma_2}).
\end{corollary}

\medskip

We can now show a global boundedness of $p(\cdot,\cdot,t)$ on $%
\mathbb{R}
^{N}\times%
\mathbb{R}
^{N}$ for each $t\in(0,\infty)$ using Proposition
\ref{Proposition_properties of p} (ii) (the Chapman--Kolmogorov equation).

\begin{theorem}
\label{Theorem}Under assumptions of condition
\ref{Condition_The main conditions} it holds%
\begin{equation}
p(x,y,t)\leq C_{N,\lambda}e^{\gamma t}t^{-\frac{N}{2}}
\label{e_main inequality}%
\end{equation}
for all $(x,y,t)\in%
\mathbb{R}
^{N}\times%
\mathbb{R}
^{N}\times(0,\infty)$, where%
\[
C_{N,\lambda}=\frac{2^{N-1}\Gamma\left(  \frac{N+1}{2}\right)  }{\pi
^{\frac{N+1}{2}}(\lambda(N-2))^{\frac{N}{2}}}%
\]
and%
\[
\gamma=-\frac{3}{4}(H_{0}^{\ast}+H_{0})\geq0\text{.}%
\]

\end{theorem}

\begin{proof}
Using H\"{o}lder's inequality and Corollary \ref{Corollary}, we obtain%
\begin{align*}
p(x,y,t)  &  =\int_{%
\mathbb{R}
^{N}}p\left(  x,z,\frac{t}{2}\right)  p\left(  z,y,\frac{t}{2}\right)  dz\\
&  \leq\left(  \int_{%
\mathbb{R}
^{N}}p\left(  x,z,\frac{t}{2}\right)  ^{2}dz\right)  ^{\frac{1}{2}}\left(
\int_{%
\mathbb{R}
^{N}}p\left(  z,y,\frac{t}{2}\right)  ^{2}dz\right)  ^{\frac{1}{2}}\\
&  \leq C_{N,\lambda}e^{\gamma t}t^{-\frac{N}{2}}\text{.}%
\end{align*}

\hfill
\end{proof}

\begin{example}
It is well known that if $A=\sum_{i=1}^{N}D_{ii}$, then%
\[
p(x,y,t)=\frac{1}{2^{N}\pi^{\frac{N}{2}}}\exp\left(  -\frac{\left\vert
x-y\right\vert ^{2}}{4t}\right)  t^{-\frac{N}{2}}\text{\qquad for }(x,y,t)\in%
\mathbb{R}
^{N}\times%
\mathbb{R}
^{N}\times(0,\infty)\text{.}%
\]
In this case we have $H_{0}=H_{0}^{\ast}=0$ and $\lambda=1$. One sees easyly
that $p$ satisfies inequality (\ref{e_main inequality}).
\end{example}

\end{document}